\documentclass[11pt]{amsart}
\usepackage{amssymb}
\usepackage{graphicx}
\usepackage{url}
\usepackage{tikz}
\usetikzlibrary{patterns}

\newcommand\restr[2]{{
  \left.\kern-\nulldelimiterspace
  #1
  \vphantom{\big|}
  \right|_{#2}
  }}
\newcommand{\suchthat}{\;\ifnum\currentgrouptype=16 \middle\fi|\;}

\newcommand{\mapsfrom}{\mathrel{\reflectbox{\ensuremath{\mapsto}}}}

\newcommand\bint[2]{{
  \left.\kern-\nulldelimiterspace
  \vphantom{\int}
  \right.^b \hspace{-0.25cm}\int_{#1}^{#2}
  }}
\newcommand\cint[2]{{
  \left.\kern-\nulldelimiterspace
  \vphantom{\int}
  \right.^c \hspace{-0.25cm}\int_{#1}^{#2}
  }}

\usepackage{enumerate}

\usepackage[latin1]{inputenc}

\newtheorem{theorem}{Theorem}
\newtheorem{proposition}[theorem]{Proposition}
\newtheorem{lemma}[theorem]{Lemma}
\newtheorem{corollary}[theorem]{Corollary}

\theoremstyle{definition}
\newtheorem{definition}[theorem]{Definition}
\newtheorem{remark}[theorem]{Remark}
\newtheorem{example}[theorem]{Example}

\begin{document}

\title{ The geometry of $E$-manifolds}

\author{Eva Miranda}
\address{Laboratory of Geometry and Dynamical Systems, Department of Mathematics and BGSMath at Universitat Polit\`ecnica de Catalunya, Barcelona and CEREMADE (Université de Paris Dauphine) -IMCCE (Observatoire de Paris) and IMJ (Université de Paris-Diderot) at Paris, France}\email{eva.miranda@upc.edu, Eva.Miranda@obspm.fr}
\thanks{Eva Miranda is supported by the Catalan Institution for Research and Advanced Studies via an ICREA Academia 2016 Prize, by a \emph{Chaire d'Excellence} de la Fondation Sciences Math\'{e}matiques de Paris, and partially supported by the  Ministerio de Econom\'{\i}a y Competitividad project with reference MTM2015-69135-P/FEDER and by the Generalitat de Catalunya project with reference 2014SGR634. This work is supported by a public grant overseen by the French National Research Agency (ANR) as part of the \emph{\lq\lq Investissements d'Avenir"} program (reference: ANR-10-LABX-0098).}
\author{Geoffrey Scott}
\address{Department of Mathematics, University of Toronto, Toronto, Canada}
\email{gscott@math.utoronto.ca}

\date{\today}

\begin{abstract}  Motivated by the study of symplectic Lie algebroids, we study a describe a type of algebroid (called an $E$-tangent bundle) which is particularly well-suited to study of singular differential forms and their cohomology. This setting generalizes the study of $b$-symplectic manifolds, foliated manifolds, and a wide class of Poisson manifolds. We generalize Moser's theorem to this setting, and use it to construct symplectomorphisms between singular symplectic forms. We give applications of this machinery (including the study of Poisson cohomology), and study specific examples of a few of them in depth.
\end{abstract}

\maketitle
\section{Introduction}

Symplectic Lie algebroids were studied by Nest and Tsygan in \cite{nestandtsygan}. In this paper we review some of their properties and give a catalogue of examples. Our main motivation for writing this paper is on the one hand to export Moser's path method to this realm but also to give explicit computations in terms of their algebroid cohomology and Poisson cohomology which are not isomorphic in general.

\begin{definition}
Let $E$ be a locally free submodule of the $C^{\infty}$ module $Vect(M)$ of vector fields on $M$. By the Serre-Swan theorem, there is an \textbf{$E$-tangent bundle} ${^E}TM$ whose sections (locally) are sections of $E$, and an \textbf{$E$-cotangent bundle} ${^E}T^*M := ({^E}TM)^*$. We will call the global sections of $\wedge^p ({^E}T^*M)$ \textbf{$E$-forms of degree $p$}, and denote the space of all such sections by ${^E}\Omega^p(M)$. If $E$ satisfies the involutivity condition $[E, E] \subseteq E$, there is a differential $d: {^E}\Omega^p(M) \rightarrow {^E}\Omega^{p+1}(M)$ given by
\begin{align*}
d\omega(V_0, \dots, V_p) &= \sum_i (-1)^iV_i\left(\omega\left(V_0, \dots, \hat{V_i}, \dots, V_p \right) \right) \\ & \ \ \ + \sum_{i < j} (-1)^{i+j}\omega\left([V_i, V_j], V_0, \dots, \hat{V}_i, \dots, \hat{V}_j, \dots, V_p\right).
\end{align*}

The cohomology of this complex is the \textbf{$E$-cohomology} ${^E}H^*(M)$. We call a closed nondegenerate $E$-form of degree 2 an \textbf{$E$-symplectic form}, and the triple $(M, E, \omega)$ an \textbf{$E$-symplectic manifold}.
\end{definition}
The bundle ${^E}TM \rightarrow M$ is naturally a Lie algebroid, whose bracket is given by the standard bracket of vector fields, and whose anchor map ${^E}TM \rightarrow TM$ is induced by the inclusion $E \subseteq Vect(M)$. The $E$-cohomology is the same as the Lie algebroid cohomology, and an \textbf{$E$-symplectic manifold} is an example of a symplectic Lie algebroid. When $E = \textrm{Vect}(M)$, then the $E$-(co)tangent bundle is the same as the standard (co)tangent bundle, and an $E$-symplectic form is a standard symplectic form. Throughout this paper, we will denote by $EVect(M)$ the set of all locally free submodules $E$ of $Vect(M)$ satisfying the involutivity condition $[E, E] \subseteq [E]$, and the \emph{rank} of an $E \in EVect(M)$ means the rank of the vector bundle ${^E}TM$. In this paper, we study examples of $E$-symplectic manifolds.
\begin{description}
\item[$b$-manifolds, $b^k$-manifolds] For any closed embedded hypersurface $Z \subseteq M$, we can let $E$ be the submodule of vector fields tangent to $Z$. In this case, the geometry of $E$-symplectic forms has been well-studied under the name of $b$-symplectic geometry (see for instance \cite{guimipi11}, \cite{guimipi12}, \cite{scott}, \cite{km}, \cite{dkm}, \cite{kms}). In the presence of the additional structure of a $k$-jet of a defining function along $Z$, the submodule $E$ of vector fields with ``order $k$ tangency'' to $Z$ defines an element of $EVect(M)$, and the resulting geometry has been studied under the name of $b^k$-geometry.
\item[$c$-manifolds] A $c$-symplectic manifold is a generalization of a $b$-symplectic manifold where the hypersurface $Z$ is allowed to have transverse self-intersection. We prove a Mazzeo-Melrose type theorem for the $E$-cohomology, and also give an explicit example of an $E$-symplectic structure of this type on $S^4$, which demonstrates that certain obstruction theorems in $b$-symplectic geometry proven in \cite{marcutosorno1}  do not generalize to the $c$-symplectic setting.
\item[Elliptic type singularities] The submodule $E \subseteq Vect(\mathbb{R}^2)$ generated (as a $C^{\infty}$-module) by the vector fields
\[
v = x \frac{\partial}{\partial x} + y \frac{\partial}{\partial y} \hspace{1cm} \textrm{and} \hspace{1cm} w = -y \frac{\partial}{\partial x} + x\frac{\partial}{\partial y}.
\]
is involutive, hence $E \in EVect(\mathbb{R}^2)$. This is the prototypical example of an \emph{elliptic type} $E$-structure, which in general consists of a codimension-2 submanifold together with a $E$-structure on its normal bundle. These have been used in \cite{cavgual1} to study stable generalized complex structures. We calculate the $E$-cohomology of such a manifold.
\item[Regular foliations] To any regular foliation of rank $r$, the submodule $E \subseteq Vect(M)$ of vector fields tangent to the leaves of the foliation is involutive by the Frobenius theorem, hence $E \in EVect(M)$. The complex $({^E}\Omega^*(M), d)$ is the familiar complex of foliated differential forms whose cohomology is foliated cohomology. In this context, Moser's lemma was proven in \cite{moserfeuillete}.
\end{description}

In addition, it is possible to construct examples by describing the generators of the submodule explictly. For example, the vector fields on $\mathbb{R}^2$
\[
v = x\frac{\partial}{\partial x} + y \frac{\partial}{\partial y} \hspace{1cm} w = y\frac{\partial}{\partial x}.
\]
generate an $E \in EVect(\mathbb{R}^2)$.

\textbf{Organization of this paper:}
The paper consists on 4 different sections. In section 2, we prove a Moser theorem for $E$-symplectic forms and discuss applications to the specific examples of $E$-symplectic geometries listed in the introduction. In section 3, we compute the $E$-cohomology of a $c$-manifold and prove a Mazzeo-Melrose formula. We also give an example of two $c$-symplectic forms on $\mathbb{R}^2$ which are not symplectomorphic, showing that the Darboux theorem of symplectic geometry does not generalize to this context. In the last section of this paper we give an explicit example of $E$-manifold and compute its $E$-cohomology proving that it is not isomorphic to its Poisson cohomology.

\section{Moser Path Method for $E$-symplectic Manifold}
In this section, we generalize Moser's theorem to the setting of $E$-symplectic manifolds. Towards this goal, we first generalize some basic operations on differential forms to the context of $E$-forms. A diffeomorphism $\phi: M \rightarrow M$  acts on a submodule $E \in EVect(M)$ by the usual pushforward, which induces a dual map ${^{\phi_*(E)}}\Omega(M) \rightarrow {^E}\Omega(M)$. Notice that unless $\phi_*(E) = E$, an $E$-form $\omega$ and its pullback $\phi^*(\omega)$ are not sections of the same bundle. For diffeomorphisms given as the flows of sections of $E$, the condition $\phi_*(E) = E$ is satisfied.

\begin{proposition}(Proposition 1.6 of \cite{as}) Let $E \in EVect(M)$, and let $X \in \Gamma(M, E)$. The time-t flow $\rho_t$ of $X$ satisfies $(\rho_t)_*(E) = E$.
\end{proposition}

\noindent We denote by $Aut(M, E)$ the set of all diffeomorphisms $\phi: M \rightarrow M$ such that $\phi_*(E) = E$.

\begin{definition}\label{def:ops} Let $E \in EVect(M)$, let $\omega \in {^E}\Omega^p(M)$, and let $X \in \Gamma(M, E)$. Define $\iota_X: {^E}\Omega^p(M) \rightarrow {^E}\Omega^{p-1}(M)$ by
\[
\iota_X(\omega)(X_1, \dots, X_{p-1}) = \omega(X, X_1, \dots, X_{p-1})
\]
for $p \geq 1$ and by the zero map for $p = 0$. Define $\mathcal{L}_{X}(\omega) := \restr{\frac{d}{dt}}{t = 0} \rho_t^*(\omega)$, where $\rho_t$ is the time-t flow of the time-dependent vector field $X_t$.
\end{definition}
\noindent These operations satisfy
\[
 \mathcal{L}_{X}(\omega) = (d\iota_X + \iota_Xd)(\omega) \hspace{1cm} \textrm{and}  \hspace{1cm} \frac{d}{dt}\rho_t^* \omega_t = \rho_t^*(\mathcal{L}_{X_t}\omega_t + \frac{d}{dt} \omega_t).
\]
The proof of this is the same as for usual differential forms (for example, see Theorem 5.2.2 of \cite{marle} and Proposition 6.4 of \cite{acs})

\begin{lemma}
Let $E \in EVect(M)$, and let $v_t$ be a time dependent $E$-vector field with flow $\varphi_t$ defined for $0 \leq t \leq 1$. Then the map ${^E}H^p(M) \rightarrow {^E}H^p(M)$ induced by $\varphi_1^*$ is the identity.
\end{lemma}

\begin{proof}
Let $\omega \in {^E}\Omega^p(M)$, and $\omega_1=\varphi_1^*(\omega)$, and consider the homotopy operator $Q: {^E}\Omega^p(M) \rightarrow {^E}\Omega^{p-1}(M)$ given by
\[
Q\omega=\int_0^1\varphi_t^*(\iota_{v_t}\omega) dt
\]
For any closed $\omega \in {^E}\Omega^p(M)$, This homotopy operator satisfies $\varphi_1^*(\omega) - \omega = dQ(\omega)$, so $[\varphi_1^*(\omega)] = [\omega]$.
\end{proof}

Beware that not every diffeomorphism in $Aut(M, E)$ isotopic to the identity is isotopic to the identity \emph{through a path in} $Aut(M, E)$.
\begin{example} Consider the submodule $E$ of $TS^2$ consisting of vector fields tangent to the equator. Let $\rho_t$ denote a rotation of $t\pi$ along the axis shown below.

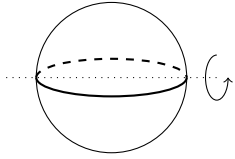
\begin{figure}[ht]\label{fig:equator_on_sphere}
\centering
\begin{tikzpicture}
\begin{scope}[xshift=-6cm]
	\draw (0, 0) circle(1cm);
	\draw[thick] (-1, 0) arc(180:360: 1cm and 0.25cm);
	\draw[thick, dashed] (-1, 0) arc(180:0: 1cm and 0.25cm);
	\draw[dotted] (-1.4, 0) -- (1.4, 0);
	\draw[->] (1.4, 0.3) arc(90:360:.15cm and .3cm);
\end{scope}
\end{tikzpicture}
\end{figure}

If $\omega = dh/h \wedge d\theta$, then $\rho_1^*(\omega) = -dh/h \wedge d\theta$, so $[\rho_1^*\omega] \neq [\omega]$. Although the map $\rho_1$ is isotopic to $\rho_0 = \textrm{id}$, it is not isotopic through a path in $Aut(M, E)$.
\end{example}

\begin{theorem}\label{thm:moser}
Let $M$ be a compact manifold, $E \in EVect(M)$, and let $\omega_0, \omega_1 \in {^E}\Omega^2(M)$. If $\omega_t := (1-t)\omega_0 + t\omega_1$ is an $E$-symplectic form for all $t \in [0, 1]$, and $[\omega_0] = [\omega_1] \in {^E}H^2(M)$, then there is a time-dependent $E$-vector field $X_t$ whose flow $\rho_t$ satisfies $\rho_t^*\omega_t = \omega_0$ for all $t \in [0, 1]$.
\end{theorem}

\begin{proof}
Let $X_t$ be a time-dependent $E$-vector field, and let $\rho_t$ be its flow. Then $\frac{d}{dt}\rho_t^* \omega_t = \rho_t^*(\mathcal{L}_{X_t}\omega_t + \frac{d}{dt} \omega_t)$, so the condition that $\rho_t^*\omega_t$ is constant is equivalent to the condition that $\mathcal{L}_{X_t}\omega_t = - \frac{d}{dt} \omega_t$. Because $\omega_t$ is closed, the Cartan formula simplifies this condition to $d(\iota_{X_t}\omega_t) = \textcolor{red}{-} \frac{d}{dt} \omega_t$. Therefore, the construction of the desired isotopy $\rho_t$ entails finding a $X_t$ satisfying this equation. Using that $\frac{d\omega_t}{dt} = \omega_1 - \omega_0$ is cohomologically trivial, pick $\mu \in\,{^E}\Omega^1(M)$ such that $d \mu = \omega_1 - \omega_0$ and define $X_t$ to be the $E$-vector field satisfying $\iota_{X_t}\omega_t = \textcolor{red}{-}\mu$. Then, $d\iota_{X_t}\omega_t = \textcolor{red}{-} \frac{d}{dt} \omega_t$ as required.
\end{proof}

We can recover the $b$-Moser theorem from \cite{guimipi12} and the $b^{m}$-Moser theorem from \cite{gmw} as a special case of Theorem \ref{thm:moser}, when the $E$-submodule in question consists of the vector fields tangent to a hypersurface of $M$ (or, in the $b^m$ case, those with a higher-order tangency to a hypersurface).
\begin{corollary}[\textbf{$b$-Moser theorem}] \label{bmoser}
Suppose that $M$ is compact and let $\omega_0$ and $\omega_1$ be two cohomologous $b$-symplectic forms on $(M,Z)$. If $\omega_t = (1-t)\omega_0 + t\omega_1$ is $b$-symplectic for $0\leq t\leq 1$, there is a family of diffeomorphisms $\gamma_t:M\to M$, for $0\leq t\leq 1$ such that $\gamma_t$ leaves $Z$ invariant and $\gamma_t^*\omega_t=\omega_0$.
\end{corollary}

\begin{theorem}\label{thm:love} Let $M$ be a compact manifold, and let $E \in EVect(M)$. If $^{E}TM$ has rank 2, and $\omega_0, \omega_1$ are two $E$-symplectic forms such that $[\omega_0] = [\omega_1]$, and $\omega_0, \omega_1$ induce the same orientation on $^{E}TM$, then there is a $\phi \in Aut(M, E)$ such that $\phi^*(\omega_1) = \omega_0$.
\end{theorem}
\begin{proof} Let $\omega_t = t\omega_1 + (1-t)\omega_0$. Then $\omega_t$ is a path of $E$-symplectic forms such that $[\omega_t]$ is constant. By Theorem \ref{thm:moser}, there is a time-dependent $E$-vector field $X_t$ whose flow $\rho_t$ satisfies $\rho_t^*\omega_t = \omega_0$. Then $\rho_1$ is the desired diffeomorphism.

\end{proof}

We can recover the classification of oriented symplectic surfaces, Radko's classification of $b$-symplectic surfaces, and the classification of $b^k$-symplectic surfaces in \cite{scott} as special cases of Theorem \ref{thm:love}.

\begin{corollary}[\textbf{Classification of symplectic surfaces, \cite{moser}}] Let $S$ be a  compact oriented surface, and and let  $\omega_0$ and $\omega_1$ be two symplectic forms on $(M,Z)$  with $[\omega_0]= [\omega_1]$. Then there exists a diffeomorphism $\phi:M\rightarrow M$ such that $\phi^*\omega_1 = \omega_0$.
\end{corollary}
\begin{corollary}[\textbf{Classification of $b$-symplectic surface, \cite{guimipi12}}] Let $S$ be a  compact orientable surface and  and let  $\omega_0$ and $\omega_1$ be two $b$-symplectic forms on $(M,Z)$  defining the same $b$-cohomology class (i.e.,$[\omega_0]= [\omega_1]$).  Then there exists a diffeomorphism $\phi:M\rightarrow M$ such that $\phi^*\omega_1 = \omega_0$.
\end{corollary}

Let $M$ be a manifold with a regular foliation, and $E$ the submodule of $Vect(M)$ consisting of vector fields tangent to the foliation. In this case, the complex ${^{E}}\Omega^{*}(M)$ is the usual complex for foliated cohomology, $H^*(\mathcal{F})$.

So in particular by rephrasing theorem \ref{thm:love} we obtain the following Moser's theorem for symplectic foliations.

\begin{theorem}\label{thm:lovefoliations} Let $M$ be a compact manifold, and let $\mathcal F$ be a regular foliation by $2$-dimensional leaves. If $\omega_0, \omega_1$ are two leafwise-symplectic forms representing the same foliated cohomology class and inducing the same orientation on $\mathcal{F}$, then there is a $\phi$ preserving the foliation  such that $\phi^*(\omega_1) = \omega_0$.
\end{theorem}

When the symplectic foliation  is the symplectic foliation of a regular Poisson manifolds, it would be desirable to obtain a path method that encodes the classification in terms of Poisson cohomology. The following result enables to do so for unimodular Poisson manifolds.
We recall from \cite{davidandeva},

\begin{theorem}[Martinez-Torres and Miranda, \cite{davidandeva}]\label{cor:1} Let $(M,\pi)$ be a compact orientable unimodular regular Poisson manifold of dimension $m$ and rank $2n$. Then
 then there is an isomorphism of cohomology groups:
 \[H^m_{\pi}(M)\rightarrow H^{2n}(\mathcal F)\]
\end{theorem}

In particular by applying theorem \ref{thm:lovefoliations} we obtain,

\begin{theorem} Let $M^3$ be a regular compact Poisson manifold of corank $1$, and let $\mathcal F$ be its symplectic foliation. If $\Pi_0, \Pi_1$ are two Poisson structures  with the same class in Poisson cohomology  and the same orientation, then there is a $\phi$ preserving the symplectic foliation  such that $\phi_*(\Pi_0) = \Pi_1$.
\end{theorem}

This global classification of $E$-symplectic objects in dimension $2$ using a Moser theorem can be used to classify Nambu structures of top degree for manifolds in any dimension as it was done in \cite{mirandaplanas} to obtain,

\begin{theorem}[Miranda and Planas, \cite{mirandaplanas}]\label{thm:bnnambu} Let $\Theta_0$ and $\Theta_1$ be two $b^m$-Nambu forms of degree $n$ on a compact orientable manifold $M^n$.
 If $[\Theta_0] = [\Theta_1]$ in $b^m$-cohomology then there exists a diffeomorphism $\phi$ such that $\phi^{*}\Theta_1 = \Theta_0$.

\end{theorem}

This classification theorem generalizes that in \cite{david}.

\section{ A case-study: c-manifolds}
In the definition of a $b$-manifold $(M, Z)$, the hypersurface $Z$ is assumed to be smoothly embedded. In this section, we generalize this theory by allowing the hypersurface to have transverse self-intersections as considered in  \cite{gmps} (see v1 in arxiv) and \cite{gualtierietal}.

For these $E$-manifolds we can compute explicitly its $E$-cohomology and prove a Mazzeo-Melrose theorem.
\subsection{Self-transverse immersions and $c$-geometry}
\begin{definition} Let $Z$ and $M$ be smooth manifolds. An immersion $i: Z \rightarrow M$ is \textbf{self-transverse} if whenever $p_1, \dots, p_n \in Z$ are distinct points mapping to the same point of $M$, the planes $\{i_*(T_{p_i}M)\}$ are in general position.
\end{definition}

\begin{definition} A self-transverse immersion $i: Z \rightarrow M$ has \textbf{multiple point manifolds}, for each $k \geq 1$, given by
\begin{align*}
Z_k &= \{(p_1, \dots, p_k) \in Z^k \mid p_i \ \textrm{distinct}, \ i(p_i) = i(p_j)\}\\
\overline{Z_k} &= Z_k / S_k
\end{align*}
where the symmetric group $S_k$ acts on $Z_k$ by permuting the $p_i$'s. For $k > \ell$, there are also self-transverse immersions
\begin{align*}
Z_k &\rightarrow Z_{\ell} \hspace{2.5cm} \textrm{and}   &Z_k &\rightarrow M\\
(p_1, \dots, p_k) &\mapsto (p_1, \dots, p_\ell)  &(p_1, \dots, p_k) &\mapsto i(p_1)
\end{align*}
These maps are equivariant with respect to the actions of the symmetric groups, hence define self-transverse immersions $\overline{Z_k} \rightarrow \overline{Z_{\ell}}$ and $\overline{Z_k} \rightarrow M$. To simplify notation, we set $Z_0 := M$.
\end{definition}

\begin{example} If $S^1$ is mapped into $\mathbb{R}^2$ as a figure eight, then there are two points $p_1, p_2 \in S^1$ that are identified under this map. Then $Z_2$ consists of two points, $(p_1, p_2), (p_2, p_1) \in S^1 \times S^1$, while $\overline{Z_2}$ consists of just one point.
\end{example}

\begin{definition}
A $c$-manifold is a manifold $M$ together with a self-transverse immersion $i$ of a hypersurface $Z$ into $M$. A vector field $v$ on $c$-manifold $M$ is a \textbf{$c$-vector field} if for each $p \in Z$, $v_{i(p)} \in i_*(T_pZ)$.
\end{definition}

The set of $c$-vector fields forms an element of $EVect(M)$, and we denote the corresponding $E$-tangent and $E$-cotangent bundles by ${^c}TM$ and ${^c}T^*M$, respectively. Around every point of $M$ in the image of $Z_k$ but not $Z_{k+1}$, there are coordinates $(x_1, \dots, x_n)$ for which the image of $Z$ is $\cup_{i \leq k} \{x_i = 0\}$. In these coordinates, $^cTM$ is generated as a $C^{\infty}$ module by
\[
\left\{ x_1\frac{\partial}{\partial x_1}, \dots,  x_k\frac{\partial}{\partial x_k}, \frac{\partial}{\partial x_{k+1}}, \dots, \frac{\partial}{\partial x_n}\right\}
\]
and $^cT^*M$ is generated by the dual basis
\[
\left\{ \frac{dx_1}{x_1}, \dots,  \frac{dx_k}{x_k}, dx_{k+1}, \dots, dx_n\right\}.
\]
In this setting, we call an $E$-symplectic form a \textbf{$c$-symplectic form}.

\begin{example}\label{ex:csymplecticsphere}
In this example, we describe a $c$-symplectic structure on $S^4$. For $1 \leq i \leq 5$, consider the following open sets of the unit sphere $S^4 \subseteq \mathbb{R}^5$:
\begin{align*}
U_i^+ &= \{(y_1, \dots, y_5) \in S^4 \mid y_i > 0\}\\
U_i^- &= \{(y_1, \dots, y_5) \in S^4 \mid y_i < 0\}
\end{align*}
A line passing through the origin in $\mathbb{R}^5$ and fixed point in $U_i^+$ intersects the plane $\mathbb{R}^4 \cong \{y_i = 1\} \subseteq \mathbb{R}^5$ in exactly one point; the projection defined in this way gives local coordinates on $U_i^+$. Likewise, the projection through the origin onto the plane $\mathbb{R}^4 \cong \{y_i = -1\} \subseteq \mathbb{R}^5$ gives local coordinates on each $U_i^-$. These coordinates are just the pullback of the standard coordinates on $\mathbb{R}P^4$ by the double covering $S^4 \rightarrow \mathbb{R}P^4$. Below are examples of two change of coordinate maps between these charts.

\begin{align*}
\left\{ (x_2, x_3, x_4, x_5) \in U_1^+ \mid x_2 > 0 \right\} \ \ \ \ \ & \left\{ (x_1, x_3, x_4, x_5) \in U_2^+ \mid x_1 > 0 \right\}\\
(x_2, x_3, x_4, x_5) \mapsto & x_2^{-1} \left( 1, x_3, x_4, x_5\right)\\
x_1^{-1}\left( 1, x_3, x_4, x_5\right) \mapsfrom &(x_1, x_3, x_4, x_5)\\
\end{align*}
\vspace{-0.5cm}
\begin{align*}
\left\{ (x_2, x_3, x_4, x_5) \in U_1^+ \mid x_2 < 0 \right\}  \ \ \ \ \ & \left\{ (x_1, x_3, x_4, x_5) \in U_2^- \mid x_1 > 0 \right\}\\
(x_2, x_3, x_4, x_5) \mapsto & -x_2^{-1}\left( 1, x_3, x_4, x_5\right)\\
x_1^{-1}\left( -1, x_3, x_4, x_5\right) \mapsfrom &(x_1, x_3, x_4, x_5)\\
\end{align*}

Let $v_j$ for $1 \leq j \leq 5$ be the vector fields on $S^4$ defined as follows on these coordinate charts.
\[
v_j = \left\{\begin{array}{r l} x_j\frac{\partial}{\partial x_j} & \textrm{on $U_i^+$ and $U_i^-$ for $i \neq j$}\\ \sum_{i \neq j} -x_i\frac{\partial}{\partial x_i} & \textrm{on $U_j^+$ and $U_j^-$} \end{array} \right.
\]

Consider the bivector $\Pi = (v_1 + v_2)\wedge(v_2 + v_3) + (v_3 + v_4) \wedge (v_4 + v_5)$ on $S^4$. Because the $v_i$ vector fields commute with one another, $[\Pi, \Pi] = 0$. Also, $\Pi \wedge \Pi = 2(\sum_{i = 1}^5 v_1 \wedge \dots \wedge \hat{v_i} \wedge \dots v_5$. Together with the fact that $v_i = -\sum_{j \neq i} v_j$, one can verify that $\Pi$ is dual to a $c$-symplectic form on each coordinate chart. For example, on $U_1^{+}$ and $U_1^-$, it is given by

\begin{align*}
\textrm{On} \  U_1^+ \ \textrm{and} \ U_1^-: \ \  &x_2x_3\frac{\partial}{\partial x_2}\wedge\frac{\partial}{\partial x_3} + x_2x_4\frac{\partial}{\partial x_2}\wedge\frac{\partial}{\partial x_4}\\
& \ \ \ \ + 2x_3x_4\frac{\partial}{\partial x_3}\wedge\frac{\partial}{\partial x_4} + x_2x_5\frac{\partial}{\partial x_2}\wedge\frac{\partial}{\partial x_5}\\
& \ \ \ \ + 2x_3x_5\frac{\partial}{\partial x_3}\wedge\frac{\partial}{\partial x_5} + x_4x_5\frac{\partial}{\partial x_4}\wedge\frac{\partial}{\partial x_5}\\
\end{align*}

This illustrates that $S^4$ admits a $c$-symplectic structure.
\end{example}
The main theorem of \cite{marcutosorno1} implies that for $n \geq 4$, any compact manifold of dimension $n$ admitting a $b$-symplectic structure must have a nontrivial class in $H^2(M)$. Example \ref{ex:csymplecticsphere} shows that a manifold may admit a $c$-symplectic structure even if it does not admit a $b$-symplectic structure.\\

To better understand ${^c}H^*(M)$, we hope for a result similar to the canonical Mazzeo-Melrose isomorphism in $b$-geometry, in which case $Z$ is embedded and
\begin{equation}\label{eqn:bmazzeomelrose}
{^b}H^p(M) \cong H^p(M) \oplus H^{p-1}(Z).
\end{equation}
Towards this goal, in the following two sections, we will generalize the definition of a \emph{Liouville volume} and the \emph{residue} of a $b$-form.

\subsection{The Liouville volume of a $c$-form}

Recall that the \emph{Liouville volume} of a $b$-form $\omega$ of top degree is defined as
\[
\lim_{\epsilon \rightarrow 0}\int_{M \backslash \{|y| \leq \epsilon\}} \omega
\]
where $y$ is a local defining function for $Z$. This number is well-defined and independent of the choice of $y$. Because the hypersurface of a $c$-manifold might not be embedded and might not be coorientable, it may have no local defining function in the traditional sense.

\begin{definition}\label{def:normalbundle} Let $W$ either be $Z_k$ for some $k \geq 1$, or a connected component of $Z_k$, and let $i: W \rightarrow M$ be its inclusion into $M$. The \textbf{normal bundle} of $W$ is the bundle $\pi: N_W \rightarrow W$ whose fiber over $p \in W$ is given by $T_{i(p)}M / i_{*}(T_pW)$.
\end{definition}
\begin{remark}\label{rmk:subbund}
Let $W$ be as in Definition \ref{def:normalbundle}, and let $p = (p_1, \dots, p_k) \in W$. For each $j$, $i_{*}(T_{p_j}Z)$ is a hyperplane in $T_{i(p)}M$ containing $i_{*}(T_pW)$, hence defines a hyperplane in $(N_W)_p$. We will denote by $H_j$ the codimension-1 subbundle of $N_W$ defined by the hyperplanes $i_*(T_{p_j}Z) \subseteq T_{i(p)}M$. Because $i: Z \rightarrow M$ is self-transverse, these $H_j$ are transverse subbundles of $N_W$.
\end{remark}
\begin{definition} Let $W$ be as in Definition \ref{def:normalbundle}. A \textbf{tubular neighborhood} of $W$ is a map $\phi: N_W \rightarrow M$ such that
\begin{itemize}
\item Around every $p \in W$, there is an open neighborhood $U \subseteq W$ such that the map $\restr{\phi}{\pi^{-1}(U)}: \pi^{-1}(U) \rightarrow M$ is a usual tubular neighborhood of $\restr{i}{U}: U \rightarrow M$
\item For each $j$, the map $\restr{\phi}{H_j}: H_j \rightarrow M$ factors through a map $H_j \rightarrow Z$ that sends the zero section over $p = (p_1, \dots, p_k)$ to $p_j \in Z$.
\end{itemize}
\end{definition}

Figure \ref{fig:normalbundle} below on the right shows a self-transverse map of a cylinder and a plane into $\mathbb{R}^3$. The double point manifold $Z_2$ consists of two circles and two lines. One if these circles is shown below, together with an illustration of the codimension-1 subbundles of $N_M$.

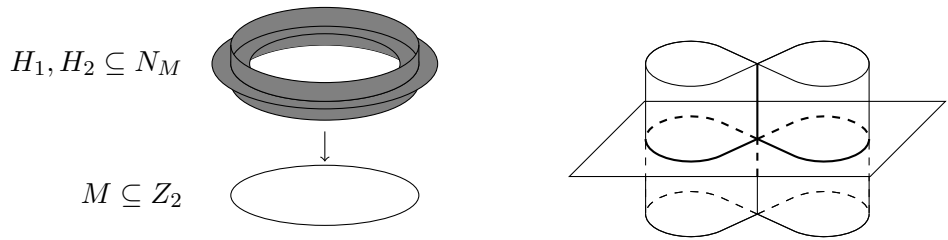
\begin{figure}[ht]
\centering
\begin{tikzpicture}

\begin{scope}[shift={(0, 1)}]
	\draw (0, 0) -- (-0.5, 0.23)  arc (50:310:0.6cm and 0.3cm) -- (0, 0);
	\draw (0, 0) -- (0.5, 0.23) arc (130:-130:0.6cm and 0.3cm) -- (0, 0);
	
	\draw (0, 0) -- (-0.5, 0.23)  arc (50:180:0.6cm and 0.3cm) -- ++(0, -1);
	\draw (0, 0) -- (0.5, 0.23)  arc (130:0:0.6cm and 0.3cm) -- ++(0, -1);
	\draw[thick] (0, 0) -- ++(0, -1);
\end{scope}

\begin{scope}[shift={(0, 0)}]
	\draw[dashed, thick] (0, 0) -- (-0.5, 0.23)  arc (50:180:0.6cm and 0.3cm);
	\draw[thick] (0, 0) -- (-0.5, -0.23) arc (310:180:0.6cm and 0.3cm);
	\draw[dashed, thick] (0, 0) -- (0.5, 0.23) arc (130:0:0.6cm and 0.3cm);
	\draw[thick] (0, 0) -- (0.5, -0.23) arc  (-130:0:0.6cm and 0.3cm);
\end{scope}

\begin{scope}[shift={(0, -1)}]
	\draw[dashed] (0, 0) -- (-0.5, 0.23)  arc (50:180:0.6cm and 0.3cm);
	\draw (0, 0) -- (-0.5, -0.23) arc (310:180:0.6cm and 0.3cm);
	\draw[dashed] (0, 0) -- (0.5, 0.23) arc (130:0:0.6cm and 0.3cm);
	\draw (0, 0) -- (0.5, -0.23) arc  (-130:0:0.6cm and 0.3cm);
	
	\draw[dashed] (0, 0) -- (-0.5, 0.23)  arc (50:180:0.6cm and 0.3cm) -- ++(0, 1);
	\draw (0, 0) -- (-0.5, -0.23) arc (310:180:0.6cm and 0.3cm) -- ++(0, 0.5);
	\draw[dashed] (0, 0) -- (0.5, 0.23)  arc (130:0:0.6cm and 0.3cm) -- ++(0, 1);
	\draw (0, 0) -- (0.5, -0.23) arc  (-130:0:0.6cm and 0.3cm) -- ++(0, 0.5);
	
	\draw (0, 0) -- ++(0, 0.5);
	\draw[dashed, thick] (0, 0.5) -- ++(0, 0.5);
\end{scope}

\draw (-2.5, -0.5) -- ++(1, 1) -- ++(4, 0) -- ++(-1, -1) -- cycle;

\begin{scope}[shift={(-6, 1)}]

	\fill[gray] (-1, -0.25) arc(-180:0:1.25cm and 0.5cm) -- ++(0, 0.25) arc(0:-180:1.25cm and 0.5cm) -- cycle;
	\fill[gray] (-1, -0.25) arc(180:0:1.25cm and 0.5cm) -- ++(0, 0.25) arc(0:180:1.25cm and 0.5cm) -- cycle;

	\draw (-1, -0.25) arc(-180:180:1.25cm and 0.5cm);
	
	\fill[gray] (-1.25, 0) arc(-180:180:1.5cm and 0.6cm) -- ++(0.5, 0) arc (180:-180:1cm and 0.4cm) -- cycle;
	\fill[gray] (-1, -0.25) arc(180:0:1.25cm and 0.5cm) -- ++(0, 0.25) arc(0:180:1.25cm and 0.5cm) -- ++(0, -0.25) -- cycle;
	
	\draw (-0.75, 0) arc(-180:180:1cm and 0.4cm);
	\draw (-1.25, 0) arc(-180:180:1.5cm and 0.6cm);
	
	\fill[gray] (-1, 0.25) arc(180:0:1.25cm and 0.5cm) -- ++(0, -0.25) arc(0:180:1.25cm and 0.5cm) -- cycle; 
	\draw (-1, 0) arc(-180:180:1.25cm and 0.5cm); 
	\fill[gray] (-1, 0.25) arc(-180:0:1.25cm and 0.5cm) -- ++(0, -0.25) arc(0:-180:1.25cm and 0.5cm) -- cycle;
	\draw (-1, 0) arc(-180:0:1.25cm and 0.5cm); 

	\draw (-1, 0.25) arc(-180:180:1.25cm and 0.5cm); 
	\draw (-1, 0) -- (-1, 0.25);
	\draw (1.5, 0) -- (1.5, 0.25);
	
	\draw[->] (0.25, -0.9) -- ++(0, -0.4);
	
	\draw (-1.5, 0) node[left] {$H_1, H_2 \subseteq N_M$} ++(0, -0.5);

\end{scope}

\begin{scope}[shift={(-6, -0.75)}]
	\draw (-1, 0) arc(-180:180:1.25cm and 0.4cm);
	\draw (-1.5, 0) node[left] {$M \subseteq Z_2$} ++(0, -0.5);
\end{scope}

\end{tikzpicture}
\caption{A normal bundle of a component of $Z_2$ with two codimension-1 subbundles}
\label{fig:normalbundle}
\end{figure}

\begin{definition}Let $(M, Z)$ be a $c$-manifold. Pick a tubular neighborhood $\phi: N_Z \rightarrow M$ of $Z$, pick an auxillary metric on $N_Z$, and let $U_{\epsilon} := \phi(\{v \in N_Z \mid |v| < \epsilon\})$. The \textbf{Liouville volume} of a compactly supported $c$-form $\omega$ of top degree is
\begin{equation} \label{eqn:louvillevolumecdensity}
\cint{M}{}\omega = \lim_{\epsilon \rightarrow 0} \int_{M \backslash U_{\epsilon}} \omega
\end{equation}
\end{definition}

\begin{proposition} \label{prop:liouvillevolumefinite} The Liouville volume of a compactly supported $c$-form is finite and independent of the choice of tubular neighborhood and auxillary metric used to define $U_{\epsilon}$
\end{proposition}
\begin{proof}
By passing to a partition of unity, it suffices to prove the claim when $\omega$ is supported in $M = [-1, 1]^n \subseteq \mathbb{R}^n$ where $i(Z) = \cup_{i \leq r} \{x_i = 0\}$. If $r = 0$, $\omega$ is a smooth form and there is nothing to prove. If $r > 0$, fix a tubular neighborhood of $Z$ and auxillary metric on $N_Z$; after a diffeomorphism of $\mathbb{R}^n$ preserving $Z$ we may assume that $U_{\epsilon} = \cup_{i = 1}^r \{|x_i| < \epsilon\}$ for small $\epsilon$. There is some $f \in C^{\infty}(\mathbb{R}^n)$ such that
\[
\omega = f \frac{dx_1}{x_1}\wedge \dots \wedge \frac{dx_r}{r} \wedge dx_{r+1} \wedge \dots \wedge dx_n
\]
By Taylor's theorem, there are smooth functions $f_i$ such that
\[
f = f(0, \dots, 0, x_{r+1}, \dots, x_n) + \sum_{i = 1}^rx_if_i.
\]
Then
\begin{align*}
\int_{M \backslash U_{\epsilon}} &f \frac{dx_1}{x_1}\wedge \dots \wedge \frac{dx_r}{r} \wedge dx_{r+1} \wedge \dots \wedge dx_n\\
&= \int_{M \backslash U_{\epsilon}} f(0, \dots, 0, x_{r+1}, \dots, x_{n}) \frac{dx_1}{x_1}\wedge \dots \wedge \frac{dx_r}{r} \wedge dx_{r+1} \wedge \dots \wedge dx_n \\
& \ \ \ + \sum_{i = 1}^r\int_{M \backslash U_{\epsilon}} f_i \frac{dx_1}{x_1}\wedge \dots \wedge dx_i \wedge \dots \wedge \frac{dx_r}{r} \wedge dx_{r+1} \wedge \dots \wedge dx_n \\
\end{align*}
The first term is finite by Fubini's theorem; the second is finite by induction on the order of the singularity of the differential form. It remains to show that this does not depend on the choice of tubular neighborhood or metric. Suppose we change the tubular neighborhood or metric just for the hyperplane $\{x_1 = 0\}$, and denote by $\widetilde{U_{\epsilon}}$ the neighborhood $U_{\epsilon}$ obtained using these new choices. The formula for $\widetilde{U_{\epsilon}}$ can be written
\[
\widetilde{U_{\epsilon}} = \{|h(x_1, \dots, x_n)| < \epsilon\} \cup \bigcup_{i = 2}^r \{|x_i| < \epsilon\}
\]
for $h$ a defining function of $\{x_1 = 0\}$. Then

\begin{align*}
\lim_{\epsilon \rightarrow 0} \int_{M \backslash \widetilde{U_{\epsilon}}} \omega &= \lim_{\epsilon_2 \rightarrow 0} \lim_{\epsilon_1 \rightarrow 0} \int_{M \backslash \{|h| < \epsilon_1\} \cup \bigcup_{i = 2}^r \{|x_i| < \epsilon_2\}} \omega\\
&= \lim_{\epsilon_2 \rightarrow 0} \lim_{\epsilon_1 \rightarrow 0} \int_{M \backslash \{|x_1| < \epsilon_1\} \cup \bigcup_{i = 2}^r \{|x_i| < \epsilon_2\}} \omega\\
&= \lim_{\epsilon \rightarrow 0} \int_{M \backslash U_{\epsilon}}\omega
\end{align*}
where the second equality is given by the independence of the Liouville volume of a $b$-form on the choice of defining function for the hypersurface. Because every change of local defining function can be written as a sequence of changes that affect only one hyperplane, this computation proves that the Liouville volume of a $c$-form is independent of the choices in its definition.
\end{proof}

Now we can define a map $\textrm{sm}: {^c}H^k(M) \rightarrow H^k(M)$ as follows. First, consider the case when $M$ is orientable. For $[\omega] \in {^c}H^k(M)$, let $\textrm{sm}([\omega]) \in H^k(M)$ be defined as the cohomology class that corresponds via Poincare duality to the element of $H^{n-k}_c(M)^*$ given by
\begin{equation}\label{eqn:smoothpart}
[\eta] \rightarrow \bint{M}{} \omega \wedge \eta\textrm{.}
\end{equation}

For the non-orientable case, we use the non-orientable Poincare duality (see Bott and Tu, Theorem 7.8), $H^k(M) \cong H^{n-k}_c(M, \textrm{or}_M)^*$, where $\textrm{or}_M$ is the orientation bundle of $M$. We also define an $E$-density on an $n$-dimensional $E$-manifold to be a section of $\Lambda^n ({^E}T^*M) \otimes \textrm{or}_M$. Just as a $c$-form restricted to the complement of $Z$ is a smooth form, a $c$-density restricted to the complement of $Z$ is a smooth density. The description of the Liouville volume of a $c$-form generalizes without any changes to the case of $c$-densities: because the proof of Proposition \ref{prop:liouvillevolumefinite} is local in nature, it applies to $c$-densities as well. So for a closed form $\omega$ on a non-orientable manifold, $\textrm{sm}([\omega]) \in H^k(M)$ is defined as the cohomology class that corresponds to the element of $H^{n-k}_c(M, \textrm{or}_M)^*$ given by equation \ref{eqn:smoothpart}

\subsection{The residue of a $c$-form}

We describe, for any $c$-manifold $(M, Z)$ and any $k \geq 0$, a \emph{residue} map
\[
\textrm{res}: {^c}\Omega^p(Z_k) \rightarrow {^c}\Omega^{p-1}(Z_{k-1})
\]
that generalizes the residue map ${^b}\Omega^p(M) \rightarrow \Omega^{p-1}(Z)$ for $b$-forms. Let $(x_1, \dots, x_n)$ be coordinates on $Z_k$ in which $\{x_1 = 0\}$ is the image of an open set $U \subseteq Z_{k+1}$ under the induced inclusion $Z_{k+1} \rightarrow Z_k$. In these coordinates, $\omega$ can be written as
\begin{equation}\label{eqn:decomp}
\omega = \alpha \wedge \frac{dx_1}{x_1} + \beta
\end{equation}

where $\alpha, \beta$ are $c$-forms of degree $p-1$ and $p$, respectively, whose coordinate expressions do not contain $\frac{dx_1}{x_1}$. On $U \subseteq Z_{k+1}$, the form $\textrm{res}(\omega)$ is defined as the pullback of $\alpha$ to $U$. Equivalently, let
\[
\mathbb{L} = (-1)^{p-1}x_1\frac{\partial}{\partial x_1}\textrm{.}
\]
Although $\mathbb{L}$ depends on the choice of coordinates, one can use the same proof in \cite{guimipi12} for the $b$-case to show that $\restr{\mathbb{L}}{i(Z_{k+1})}$ does not depend on the coordinates, and that the pullback $i^*(\iota_{\mathbb{L}}\omega)$ to $Z_k$ is a well-defined $c$-form on $Z_k$ which also equals $\textrm{res}(\omega)$. Also note that the residue map can be composed with itself; if $\omega \in {^c}\Omega^p(M)$, then $\textrm{res}^k(\omega) \in {^c}\Omega^{p-k}(Z_k)$.

\begin{example}\label{ex:mpm}

The form $\omega = \frac{dx}{x} \wedge \frac{dy}{y} \wedge \frac{dz}{z}$ is a $c$-form on $\mathbb{R}^3$, where $Z$ consists of the inclusion of the three coordinate hyperplanes. Its residue is written in Figure \ref{fig:mpmresidues}.

\begin{figure}[ht]
\centering
\begin{tikzpicture}[scale = .7]

\draw[thick, red, dashed] (0, 1.5) -- (0, -1.5);
\draw[thick, red, dashed] (-1.5, 0) -- (1.5, 0);
\draw[thick, red, dashed] (-0.5, -0.5) -- (0.5, 0.5);

\draw[thick, dashed, fill = gray, opacity = 0.5] (-0.5, -2) -- (0.5, -1) -- (0.5, 2) -- (-0.5, 1) -- cycle;
\draw[thick, dashed, fill = gray, opacity = 0.5] (-2, -0.5) -- (1, -0.5) -- (2, 0.5) -- (-1, 0.5) -- cycle;
\draw[thick, dashed, fill = gray, opacity = 0.5] (-1.5, -1.5) -- (-1.5, 1.5) -- (1.5, 1.5) -- (1.5, -1.5) -- cycle;


\draw[thick, red] (0, 1.5) -- (0, 0); \draw[thick, red] (0, -0.5) -- (0, -1.5);
\draw[thick, red] (-1.5, 0) -- (-0.5, 0); \draw[thick, red] (0, 0) -- (1.5, 0);
\draw[thick, red] (-0.5, -0.5) -- (0, 0);

\draw[thick] (0, -1.5) -- (-0.5, -2) -- (-0.5, 1) -- (0.5, 2) -- (0.5, 1.5);
\draw[thick] (-1.5, 0) -- (-2, -0.5) -- (1, -0.5) -- (2, 0.5) -- (1.5, 0.5);
\draw[thick] (-1.5, -1.5) -- (-1.5, 1.5) -- (1.5, 1.5) -- (1.5, -1.5) -- cycle;

\draw[thick, red, scale = 0.5, fill=red] (0, 0) circle(0.8mm);


\draw[thick, fill = gray, opacity = 0.5, shift={(-4.25,  1.35)}, scale = 0.5] (-0.5, -2) -- (0.5, -1) -- (0.5, 2) -- (-0.5, 1) -- cycle;
\draw[thick, fill = gray, opacity = 0.5, shift={(-4.25,  0)}, scale = 0.5] (-2, -0.5) -- (1, -0.5) -- (2, 0.5) -- (-1, 0.5) -- cycle;
\draw[thick, fill = gray, opacity = 0.5, shift={(-4.25, -1.35)}, scale = 0.5] (-1.5, -1.5) -- (-1.5, 1.5) -- (1.5, 1.5) -- (1.5, -1.5) -- cycle;

\draw[thick, red, shift={(-4.25,  1.35)}, scale = 0.5] (0, 1.5) -- (0, -1.5);
\draw[thick, red, shift={(-4.25,  1.35)}, scale = 0.5] (-0.5, -0.5) -- (0.5, 0.5);
\draw[thick, red, shift={(-4.25,  1.35)}, scale = 0.5, fill=red] (0, 0) circle(0.8mm);

\draw[thick, red, shift={(-4.25,  0)}, scale = 0.5] (-0.5, -0.5) -- (0.5, 0.5);
\draw[thick, red, shift={(-4.25,  0)}, scale = 0.5] (-1.5, 0) -- (1.5, 0);
\draw[thick, red, shift={(-4.25,  0)}, scale = 0.5, fill=red] (0, 0) circle(0.8mm);

\draw[thick, red, shift={(-4.25,  -1.35)}, scale = 0.5] (-1.5, 0) -- (1.5, 0);
\draw[thick, red, shift={(-4.25,  -1.35)}, scale = 0.5] (0, -1.5) -- (0, 1.5);
\draw[thick, red, shift={(-4.25,  -1.35)}, scale = 0.5, fill=red] (0, 0) circle(0.8mm);

\draw[thick, shift={(-4.25,  1.35)}, scale = 0.5] (-0.5, -2) -- (0.5, -1) -- (0.5, 2) -- (-0.5, 1) -- cycle;
\draw[thick, shift={(-4.25,  0)}, scale = 0.5] (-2, -0.5) -- (1, -0.5) -- (2, 0.5) -- (-1, 0.5) -- cycle;
\draw[thick, shift={(-4.25, -1.35)}, scale = 0.5] (-1.5, -1.5) -- (-1.5, 1.5) -- (1.5, 1.5) -- (1.5, -1.5) -- cycle;

\draw[shift = {(-4.25, -2.75)}, scale = 0.5] node{$Z = Z_1$};


\draw[thick, red, shift={(-7.5,  1.35)}, scale = 0.5, <->] (0, 1.5) -- (0, -1.5);
\draw[thick, red, shift={(-8.5,  1.35)}, scale = 0.5, <->] (-0.5, -0.5) -- (0.5, 0.5);
\draw[thick, red, shift={(-7.5,  1.35)}, scale = 0.5, fill=red] (0, 0) circle(0.8mm);
\draw[thick, red, shift={(-8.5,  1.35)}, scale = 0.5, fill=red] (0, 0) circle(0.8mm);

\draw[thick, red, shift={(-7.5,  0)}, scale = 0.5, <->] (-0.5, -0.5) -- (0.5, 0.5);
\draw[thick, red, shift={(-8.5,  0)}, scale = 0.5, <->] (-1.5, 0) -- (1.5, 0);
\draw[thick, red, shift={(-7.5,  0)}, scale = 0.5, fill=red] (0, 0) circle(0.8mm);
\draw[thick, red, shift={(-8.5,  0)}, scale = 0.5, fill=red] (0, 0) circle(0.8mm);

\draw[thick, red, shift={(-7.5,  -1.35)}, scale = 0.5, <->] (-1.5, 0) -- (1.5, 0);
\draw[thick, red, shift={(-8.5,  -1.35)}, scale = 0.5, <->] (0, -1.5) -- (0, 1.5);
\draw[thick, red, shift={(-7.5,  -1.35)}, scale = 0.5, fill=red] (0, 0) circle(0.8mm);
\draw[thick, red, shift={(-8.5,  -1.35)}, scale = 0.5, fill=red] (0, 0) circle(0.8mm);

\draw[shift = {(-8, -2.75)}, scale = 0.5] node{${Z_2}$};


\draw[thick, red, shift={(-11,  1.35)}, scale = 0.5, fill=red] (0, 0) circle(0.8mm);
\draw[thick, red, shift={(-12,  1.35)}, scale = 0.5, fill=red] (0, 0) circle(0.8mm);

\draw[thick, red, shift={(-11,  0)}, scale = 0.5, fill=red] (0, 0) circle(0.8mm);
\draw[thick, red, shift={(-12,  0)}, scale = 0.5, fill=red] (0, 0) circle(0.8mm);

\draw[thick, red, shift={(-11,  -1.35)}, scale = 0.5, fill=red] (0, 0) circle(0.8mm);
\draw[thick, red, shift={(-12,  -1.35)}, scale = 0.5, fill=red] (0, 0) circle(0.8mm);

\draw[shift = {(-11.5, -2.75)}, scale = 0.5] node{${Z_3}$};

\end{tikzpicture}

\caption{The components of ${Z_1}$ are $\{x=0\}, \{z=0\},$ and $\{y=0\}$, drawn from top to bottom. The diagram of ${Z_2}$ (respectively, ${Z_3}$) is arranged so that the canonical maps $Z_2 \rightarrow Z_1$ send components of $Z_2$ to components of $Z_1$ drawn to their right.}
\label{fig:mpm}
\end{figure}
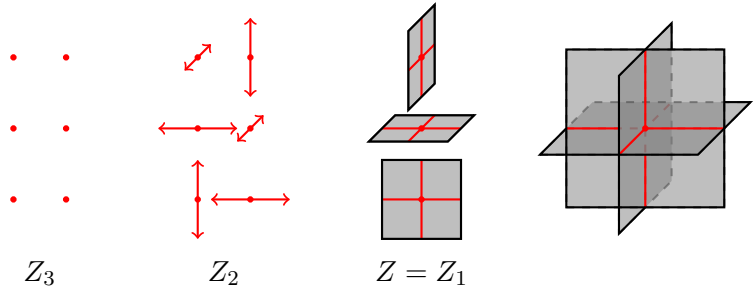

\begin{figure}[ht]
\centering
\begin{tikzpicture}[scale = .7]
\draw[thick, shift={(-0.75,  0)}, scale = 0.5] node{$\frac{dx}{x} \wedge \frac{dy}{y}\wedge \frac{dz}{z}$};

\draw[thick, shift={(-4.25,  1.35)}, scale = 0.5] node{$\frac{dy}{y}\wedge \frac{dz}{z}$};

\draw[shift={(-4.25,  0)}, scale = 0.5] node{$\frac{dx}{x} \wedge \frac{dy}{y}$};

\draw[thick, shift={(-4.25,  -1.35)}, scale = 0.5] node{$\frac{dz}{z} \wedge \frac{dx}{x}$};


\draw[thick, shift={(-7.25,  1.35)}, scale = 0.5] node{$-\frac{dz}{z}$};
\draw[thick, shift={(-8.75,  1.35)}, scale = 0.5] node{$\frac{dy}{y}$};

\draw[thick, shift={(-7.25,  0)}, scale = 0.5] node{$-\frac{dy}{y}$};
\draw[thick, shift={(-8.75,  0)}, scale = 0.5] node{$\frac{dx}{x}$};

\draw[thick, shift={(-7.25,  -1.35)}, scale = 0.5] node{$-\frac{dx}{x}$};
\draw[thick, shift={(-8.75,  -1.35)}, scale = 0.5] node{$\frac{dz}{z}$};

\draw[thick, shift={(-11,  1.35)}, scale = 0.5] node{$-1$};
\draw[thick, shift={(-12.5,  1.35)}, scale = 0.5] node{$1$};

\draw[thick, shift={(-11,  0)}, scale = 0.5] node{$-1$};
\draw[thick, shift={(-12.5,  0)}, scale = 0.5] node{$1$};

\draw[thick, shift={(-11,  -1.35)}, scale = 0.5] node{$-1$};
\draw[thick, shift={(-12.5,  -1.35)}, scale = 0.5] node{$1$};

\end{tikzpicture}

\caption{The residue of $\omega$, graphically arranged to match Figure \ref{fig:mpm}.}
\label{fig:mpmresidues}
\end{figure}

\end{example}

\begin{definition} Let $\omega \in {^c}\Omega^*({Z}_k)$. We say that $\omega$ is \textbf{compatible} if the following two conditions hold
\begin{itemize}
\item The action of any $\sigma \in S_k$ satisfies $\sigma^*(\omega') = \textrm{sign}(\sigma)\omega'$, where $\textrm{sign}(\sigma)$ is 1 if $\sigma$ corresponds to an even permutation, and $-1$ otherwise.
\item Either $\omega$ is smooth, or $\textrm{res}(\omega)$ is compatible.
\end{itemize}
\end{definition}
Because taking the repeated residue of a form will eventually give a smooth form, this recursive definition makes sense.

By the nature of the residue map, any $c$-form on a manifold $M$ will be compatible, but for a $c$-form on ${Z}_k$ to be compatible, it must ``agree'' at the points where $Z_k$ is identified in $M$, as illustrated in the example below.

\begin{example} Let $i: \mathbb{R} \sqcup \mathbb{R} \rightarrow \mathbb{R}^2$ be given by $x \mapsto (x, 0)$ on the first copy of $\mathbb{R}$ and $y \mapsto (0, y)$ on the second. Then an element of ${^c}\Omega^1({Z}_1)$ which is described by $f\frac{dx}{x}$ and $g\frac{dy}{y}$ is \emph{compatible} if $f(0) = -g(0)$.
\end{example}

Let $\widehat{\Omega}^*({Z}_i)$ (and ${^c}\widehat{\Omega}^*({Z_i})$) denote the space of compatible smooth forms (and $c$-forms). The differential preserves the property of compatibility; we write $\widehat{H}^*({Z_i})$ (and ${^c}\widehat{H}^*({Z_i})$) for the cohomology of this subcomplex. For the remainder of this section, we will use the notation $Z_0$ to denote $M$.

\begin{proposition} Let $(M, Z)$ be a $c$-manifold. Then for all $i \geq 0$,
\begin{equation}\label{eqn:sesomega}
0 \rightarrow \widehat{\Omega}^p(Z_i) \rightarrow {^c}\widehat{\Omega}^p(Z_i) \xrightarrow{\textrm{res}} {^c}\widehat{\Omega}^{p-1}(Z_{i+1}) \rightarrow 0
\end{equation}
is a short exact sequence of chain complexes. Moreover, any closed form in ${^c}\widehat{\Omega}^{p-1}(Z_{i+1})$ has a closed form in its preimage in ${^c}\widehat{\Omega}^{p}(Z_{i})$.
\end{proposition}
\begin{proof}
In verifying that Equation \ref{eqn:sesomega} is a short exact sequence of chain maps, the nontrivial parts are to verify that $\textrm{res}$ commutes with the differentials, and is surjective. First, let $\omega \in {^c}\widehat{\Omega}^p(Z_i)$. Around a point of $i(Z_{i+1}) \subseteq Z_i$, we can pick local coordinates and use the decomposition of $\omega$ described in Equation \ref{eqn:decomp},
\[
\omega = \alpha \wedge \frac{dx_1}{x_1} + \beta.
\]
Then $d\omega = d\alpha \wedge \frac{dx_1}{x_1} + d\beta$, whose residue on $Z_{i+1}$ is given by the pullback of $d\alpha$. The fact that $\textrm{res}$ commutes with $d$ then follows from the fact that differentials commute with pullbacks.\\

To verify that $\textrm{res}$ is surjective, let $\alpha \in {^c}\hat{\Omega}^{p-1}(Z_{i+1})$. We will construct an $\omega \in {^c}\Omega^p(Z_i)$ whose residue is $\alpha$. In the case that $\alpha$ is closed, this construction will produce a closed $\omega$, proving the final part of the theorem.\\

The $S_{i+1}$ action on $Z_{i+1}$ acts on the components of $Z_{i+1}$. Let $W$ be a union of components of $Z_{i+1}$, one from each orbit of this action. Let $\pi: N_{W} \rightarrow W$ be the tubular neighborhood of $W$ in $Z_i$. If $W \subseteq Z_i$ is coorientable, let $x_1 \in C^{\infty}(N_W)$ be a defining function for the zero section of $N_W$, and let $\tilde{x_i} = \rho \circ x$, where $\rho: \mathbb{R} \rightarrow \mathbb{R}$ be an odd function which is locally constant outside $[-1, 1]$ and equal to $\rho(x) = x$ on $[\frac{1}{2}, \frac{1}{2}]$.

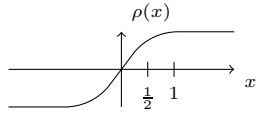
\begin{figure}[ht]
\centering
\begin{tikzpicture}

\draw[->] (-1.5, 0) -- (1.5, 0) node[below right] {\tiny $x$};
\draw[->] (0, -.5) -- (0, 0.5) node[above right] {\tiny $\rho(x)$};

\draw (-1.5, -0.5) -- (-0.75, -0.5) arc (-90:-40:0.75cm) -- (0, 0);
\draw (1.5, 0.5) -- (0.75, 0.5) arc (90:140:0.75cm) -- (0, 0);

\draw (0.35, 0.1) -- (0.35, -0.1) node[below] {\tiny $\frac{1}{2}$};
\draw (0.7, 0.1) -- (0.7, -0.1) node[below] {\tiny $1$};

\end{tikzpicture}
\end{figure}
Then, let $\eta = \pi^*(\alpha_k) \wedge \frac{d\tilde{x_1}}{\tilde{x_1}}$. In the case when $W \subseteq Z_i$ is not coorientable, then a defining function $x_1$ can only be defined up to sign. But because $\rho$ is an odd function, $d\tilde{x_1}/\tilde{x_1}$ does not change if $x_1$ is replaced by $-x_1$, so the construction of $\eta$ is well-defined even when $W \subseteq Z_i$ is not coorientable.\\

Then $\eta$ is a $c$-form on $N_M$ whose restriction to each fiber is compactly supported, $\textrm{res}(\eta) = \alpha$, and $\eta$ is closed if $\alpha$ is. Next, let $\phi: N_W \rightarrow M$ be a tubular neighborhood of $W$. If $\phi$ is an embedding, it is a diffeomorphism onto its image, and we could construct a form $\omega'$ in a neighborhood of $W \subseteq Z_{i}$ simply by identifying $N_M$ with its image and taking $\omega'$ to be $\eta$. If $\phi$ is not an embedding (for example, if $W = S^1$ is immersed as a figure-eight), then for every open set $U \subseteq Z_i$ for which $\restr{\phi}{\phi^{-1}(U)}$ is a $k$-to-one cover, then define $\restr{\omega'}{U}$ to be the \emph{sum} of the $k$ different forms obtained by identifying $U$ with each of the $k$ different preimages, and taking $\eta$ in each preimage.

Finally, let $\omega$ equal $\omega'$ summed over the action of $S_k$. By construction, $\omega$ is a compatible form with $\textrm{res}(\omega) = \alpha$. Because $\eta$ is closed whenever $\alpha$ is closed, the same is true for $\omega'$ and for $\omega$.

\end{proof}

\begin{theorem}\label{thm:c-cohomology}
Let $(M, Z)$ be a $c$-manifold. Then

\begin{align*}
{^c}H^p(M) &\cong  H^p(M) \oplus \widehat{H}^{p-1}(Z_1) \oplus \widehat{H}^{p-2}(Z_2) \oplus \cdots \oplus \widehat{H}^{0}(Z_p)\\
[\omega] &\mapsto (\textrm{sm}(\omega), \textrm{sm}(\textrm{res}(\omega)), \textrm{sm}(\textrm{res}^2(\omega)), \dots, \textrm{sm}(\textrm{res}^p(\omega)))
\end{align*}
\end{theorem}
\begin{proof}
For each $i$, the short exact sequence in Equation \ref{eqn:sesomega} induces short exact sequences in cohomology
\[
0 \rightarrow \widehat{H}^p(Z_i) \rightarrow {^c}\widehat{H}^p(Z_i) \xrightarrow{-\textrm{res}} {^c}\widehat{H}^{p-1}(Z_{i+1}) \rightarrow 0.
\]
which is split by the map
\[
\textrm{sm}: {^c}\widehat{H}^p(Z_i) \rightarrow \widehat{H}^p(Z_i).
\]
Applying induction on $i$, and the fact that ${^c}\widehat{\Omega}^p(Z_0) = {^c}\widehat{\Omega}^p(M) = {^c}\Omega^p(M)$ yields the result.
\end{proof}

Notice that each $\widehat{H}^p(Z_i)$ is isomorphic to $H^p(\overline{Z}_i)$, but that this isomorphism is only canonical up to a choice of sign.

Theorem \ref{thm:c-cohomology} implies that unlike usual symplectic forms (for which there is a Darboux theorem), $c$-symplectic forms have local invariants.

\begin{example}\label{ex:nodarboux}
Consider the family $k\frac{dx}{x}\wedge \frac{dy}{y}$ of $c$-symplectic forms on $\mathbb{R}^2$, parametrized by $k > 0$. The terms of the residue decomposition of $[k\frac{dx}{x}\wedge \frac{dy}{y}]$ from Theorem \ref{thm:c-cohomology} all vanish except for the part on $Z_2$. On $Z_2$, the residue is given by the locally constant function equal to $k$ on one of the components of $Z_2$, and $-k$ on the other component. Therefore, each $c$-symplectic form in this family represents a different $c$-cohomology class, and hence the $c$-symplectic manifolds in this family are pairwise non-symplectomorphic.
\end{example}

Example \ref{ex:nodarboux} shows that there may be cohomological obstructions to constructing local symplectomorphisms between $c$-symplectic forms; the following theorem shows that these are the \emph{only} such obstructions.

\begin{theorem}[$c$-Darboux theorem]
Let $\omega_0$ and $\omega_1$ be two $c$-symplectic forms for which the three elements under the isomorphism in theorem \ref{thm:c-cohomology} above coincide, then $\omega_0$ and $\omega_1$ are symplectomorphic.
\end{theorem}

\begin{proof}From theorem  \ref{thm:c-cohomology}, the three classes determine the cohomology class in ${^c}H^p(M)$ thus from Moser's theorem \cite{moser} we obtain the desired result.
\end{proof}

\begin{remark} As a consequence of the theorem above, there is not a unique model for $c$-symplectic structures due to the fact that the residue  in theorem \ref{thm:c-cohomology} is not constant. 

\end{remark}
\section{$E$-cohomology and Poisson cohomology}

In this section we focus our attention on comparing $E$-cohomology with Poisson cohomology associated to Poisson structures whose associated algebroid is $E$. In \cite{guimipi12} it was proven that for $b$-symplectic manifolds, $b$-cohomology is isomorphic to Poisson cohomology.

In this section we give an example proving that this is not true in general for $E$-symplectic manifolds. In this section, $E$ refers to the
 locally free involutive submodule of $Vect(\mathbb{R}^2)$ generated by
\[
v_1 = x\frac{\partial}{\partial x} + y\frac{\partial}{\partial y} \hspace{1cm} \textrm{and} \hspace{1cm} v_2 = -y\frac{\partial}{\partial x} + x\frac{\partial}{\partial y}
\]
Then ${^E}T^*M$ is generated by the sections
\[
v_1^* = \frac{xdx + ydy}{x^2 + y^2} \hspace{1cm} \textrm{and} \hspace{1cm} v_2^* = \frac{-ydx + xdy}{x^2 + y^2}
\]

We can naturally associate a Poisson structure to $E$, $\Pi_E=v_1\wedge v_2= (x^2+y^2) \frac{\partial}{\partial x}\wedge \frac{\partial}{\partial y}$.

We start by computing $E$-cohomology.

\subsection{A computation of $E$-cohomology }
Recall that a function $f \in C^{\infty}(\mathbb{R}^n)$ is called \textbf{analytically flat} at a point $p$ if $f$ and all its derivatives vanish at $p$.

\begin{lemma} Let $f(x, y) \in C^{\infty}(\mathbb{R}^2)$ be analytically flat at $0$ and $g(x, y) \in C^{\infty}(\mathbb{R}^2)$ be a nonnegative function. For any positive even numbers $p, q$, the function
\[
h(x, y) = \frac{f(x, y)}{x^p + y^q + g(x, y)}
\]
\noindent is smooth and  analytically flat.
\end{lemma}
\begin{proof} Because $x^p + y^q + g(x, y)$ is nonzero away from $0$, it suffices to show that $h$ is defined and smooth at $0$. To show that it is defined, assume without loss of generality that $p \geq q$. By Taylor's theorem, we can write
\[
f = \sum_{\substack{i, j \geq 0 \\ i+j = 2p}} f_{ij}x^iy^j
\]
\noindent for flat functions $f_{ij}$. Notice that
\begin{align*}
|h| &= \left| \sum_{i + j = 2p} \frac{f_{ij}x^iy^j}{x^p + y^q + g} \right| \\
&= \left| \sum_{\substack{i + j = 2p\\ i \geq p}} f_{ij}x^{i-p}y^j \frac{x^p}{x^p + y^q + g} \right|  + \left| \sum_{\substack{i + j = 2p\\ i < p}} f_{ij}x^{i}y^{j-q} \frac{y^q}{x^p + y^q + g} \right| \\
&\leq \left| \sum_{\substack{i + j = 2p\\ i \geq p}} f_{ij}x^{i-p}y^j \right|  + \left| \sum_{\substack{i + j = 2p\\ i < p}} f_{ij}x^{i}y^{j-q} \right| \\
\end{align*}
By the sandwich theorem, the limit exists. To prove that $h$ is smooth, observe that the first partial derivatives of $h$ are again a quotient of a flat function by a function of the form $(x^{2p} + y^{2q} + \tilde{g})$ and apply induction.
\end{proof}

\begin{theorem} Let $E \subseteq Vect(M)$, $v_1, v_2$ be as described above.
\[
{^E}H^k(M) = \left\{ \begin{array}{l l} \mathbb{R} & k = 0\\ \mathbb{R}^2 & k = 1\\ \mathbb{R} & k = 2\\ 0 & \textrm{otherwise} \end{array} \right.
\]
\end{theorem}
\begin{proof}
The degree 0 cohomology is generated by the classes of the constant functions, and hence is isomorphic to $\mathbb{R}$. Let $K$ be the closed $E$-forms of degree 1, and consider the map
\[
K \rightarrow \mathbb{R}^2 \hspace{2cm}
gv_1^* + hv_2^* \mapsto (g(0), h(0)).
\]
Exact 1-forms $v_2(f)v_1^* + v_1(f)v_2^*$ are in the kernel of this map, so this induces ${^E}H^1(\mathbb{R}^2) \rightarrow \mathbb{R}^2$. The map is surjective, since $c_1v_1^* + c_2v_2^* \mapsto (c_1, c_2)$ for any $c_1, c_2$. To show that it is surjective, let $\eta = gv_1^* + hv_2^*$ be in the kernel of this map; we will show that $\eta$ is exact. Let $\sum_{i, j \geq 0} g_{ij}x^iy^j$ be the analytic jet of $g$. Then
\[
f = \sum_{\substack{i, j \geq 0\\ (i, j) \neq (0, 0)}} (i+j)^{-1}x^iy^j
\]
has the property that $v_1(f) - g$ is flat. Because $\eta$ and $df = v_1(f)v_1^* + v_2(f)v_2^*$ are closed,
\begin{align*}
0 = d\eta - d^2f &= (v_1(h) - v_2(g) - v_1v_2(f) + v_2v_1(f))v_1^*\wedge v_2^*\\
&= (v_1(h - v_2(f)) - v_2(g - v_1(f))) v_1^*\wedge v_2^*\\
\end{align*}
Because $v_1(f) - g$ is flat, so is $v_2(g - v_1(f))$, and it therefore follows that $v_1(h - v_2(f))$ is flat and therefore $h - v_2(f)$ is analytically a constant function. Since $h(0) = 0$, $h - v_2(f)$ is flat. Notice that
\[
(g - v_2(f))v_1^* + (h - v_1(f))v_2^* = \frac{g - v_2(f)}{x^2 + y^2}(xdx + ydy) + \frac{h - v_1(f)}{x^2 + y^2}(-ydx + xdy)
\]
is smooth by the lemma. It is also closed, so by Poincare's lemma has a primitive $F$. Then $f + F$ is a primitive for $\eta$. To calculate ${^E}H^2(\mathbb{R}^2)$, let $K$ be the closed $E$-forms of degree 2, and consider
\[
K \rightarrow \mathbb{R} \hspace{2cm}
gv_1^*\wedge v_2^* \mapsto g(0).
\]
By a similar argument as before, this map is surjective and contains all exact forms in its kernel. Let $\eta = gv_1^* \wedge v_2^*$ be a general element of the kernel, and define $f$ in the same way as above so that $v_1(f) - g$ is flat. Then $\eta - df$ is a closed smooth form, and hence has a primitive $F$. Then $f + F$ is a primitive for $\eta$.
\end{proof}

\subsection{A Poisson cohomology computation for $ E$}

The Poisson cohomology associated to  the Poisson structure $\Pi_E=v_1\wedge v_2= (x^2+y^2) \frac{\partial}{\partial x}\wedge \frac{\partial}{\partial y}$ on $\mathbb R^2$ was computed by Abreu and Ginzburg and an explicit computation is available at \cite{ginzburg} (Proposition 2.16).

\begin{theorem}[Abreu-Ginzburg, \cite{ginzburg}]
The Poisson cohomology groups of $(\mathbb R^2,\Pi_E)$  are as follows:
\begin{enumerate}
\item $H^0_{\Pi_E}(\mathbb R^2)=\mathbb R$ is given by constant functions.
\item $H^1_{\Pi_E}(\mathbb R^2)=\mathbb R^2$ is generated, by the rotation $x\frac{\partial }{\partial y} - y\frac{\partial }{\partial x}$, and by the dilation $x\frac{\partial }{\partial x} + y\frac{\partial }{\partial y}$.
\item  $H^2_{\Pi_E}(\mathbb R^2)=\mathbb R^2$ generated by $\frac{\partial}{\partial x}\wedge \frac{\partial}{\partial y}$ and $\Pi_E$.
\end{enumerate}
\end{theorem}

In particular this proves (due to the second cohomology group computation) that $E$-cohomology is not isomorphic to Poisson cohomology.

\begin{corollary} For the particular choice of $E$ as above and $M=\mathbb R^2$,
  $E$-cohomology is not isomorphic to Poisson cohomology.

\end{corollary}

\end{document}